\documentclass[11pt]{article}
\usepackage{a4wide,amsmath,amsbsy,amsfonts,amssymb,
stmaryrd,amsthm,mathrsfs,graphicx, amscd, tikz-cd, cancel, supertabular}
\usepackage[normalem]{ulem}
\usetikzlibrary{matrix,arrows,decorations.pathmorphing}
\usepackage{tabularx}

\usepackage{tikz}
\usetikzlibrary{matrix,arrows,decorations.pathmorphing}
\usepackage[all,cmtip]{xy}
\usepackage{qtree}

\usepackage[top=1.2in,bottom=1.2in,left=1.21in,right=1.21in]{geometry}
\usepackage{hyperref}

\newtheorem{theorem}{Theorem}[section]

\newtheorem{corollary}[theorem]{Corollary}
\newtheorem{lemma}[theorem]{Lemma}

\theoremstyle{definition}

\numberwithin{equation}{section}

\theoremstyle{definition}
\newtheorem{remark}[theorem]{Remark}

\newcommand{\Cb}{\mathbb{C}}

\newcommand{\Pb}{\mathbb{P}}

\newcommand{\Zb}{\mathbb{Z}}

\newcommand{\beq}{\begin{eqnarray}}
\newcommand{\eeq}{\end{eqnarray}}

\DeclareMathOperator{\Aut}{Aut}

\DeclareMathOperator{\IJac}{J}

\title{Irrationality of the general smooth quartic $3$-fold using intermediate Jacobians}
\author{Benson Farb\thanks{Supported in part by National Science Foundation Grant No. DMS-181772 and the Eckhardt Faculty Fund.}}

\begin{document}

\maketitle
\begin{abstract}
We prove that the intermediate Jacobian of the Klein quartic $3$-fold $X$ 
is not isomorphic, as a principally polarized abelian variety, to a product of Jacobians of curves.  As corollaries we deduce (using a criterion of Clemens-Griffiths) that $X$, as well as the general smooth quartic 
$3$-fold, is irrational.  These corollaries were known: Iskovskih-Manin \cite{IM} proved that 
every smooth quartic $3$-fold is irrational.  However, the method of 
proof here is different than that of \cite{IM}, is significantly simpler, and produces an explicit example.

\end{abstract}

\section{Introduction}

A {\em smooth quartic $3$-fold} is a smooth, degree $4$ 
hypersurface $Y$ in complex projective space $\Pb^4$.  For such a $Y$ there is a Hodge decomposition \[H^3(Y;\Cb)=H^{2,1}(Y)\oplus H^{1,2}(Y)\] 
and an attached {\em intermediate Jacobian} 
\[\IJac(Y):=\frac{H^{1,2}(Y)}{\displaystyle i(H_3(Y;\Zb))}\]
where the embedding 
\[i:H_3(Y;\Zb) \to H^{2,1}(Y)^*\cong H^{1,2}(Y)\] is defined 
by sending $\alpha\in H_3(Y;\Zb)$ to the linear functional $\omega\mapsto\int_\alpha\omega$, followed by the isomorphism 
$H^{2,1}(Y)^*\to H^{1,2}(Y)$ given by the intersection form on $Y$.    The complex torus $\IJac(Y)$ is a $30$-dimensional abelian variety.  It has a principal polarization defined by the Hermitian form 
\[Q(\alpha,\beta):=2i\int_Y \alpha\wedge\bar{\beta}.\]

The {\em Klein quartic $3$-fold} $X$  is the smooth, degree $4$ hypersurface 
\[X:=\{[x_0:\cdots :x_4]: x_0^3x_1+x_1^3x_2+x_2^3x_3 +x_3^3x_4+x_4^3x_0=0\}\subset\Pb^4.\] 
$X$ admits a non-obvious faithful action of $\Zb/61\Zb\rtimes \Zb/5\Zb$ by 
automorphisms; see \S\ref{section:proof}.  We will use these symmetries to prove the following.

\begin{theorem}[{\bf Intermediate Jacobian}]
\label{Theorem:IJ1}
The intermediate Jacobian $\IJac(X)$ of the Klein quartic $3$-fold $X$ is not isomorphic, as a principally polarized abelian variety, to a product of Jacobians of smooth curves.
\end{theorem}

A short argument using resolution of singularities (Corollary 3.26 of \cite{CG}) gives the {\em Clemens-Griffiths criterion} : if $Y$ is rational then $\IJac(Y)$ is isomorphic as a principally polarized abelian variety (henceforth p.p.a.v.) to a product of Jacobians of smooth curves.   Theorem~\ref{Theorem:IJ1} thus implies:

\begin{corollary}[{\bf Irrationality of Klein}]
\label{cor:Klein1}
The Klein quartic $3$-fold is {\em irrational}: it is not birational to $\Pb^3$.
\end{corollary}

The intermediate Jacobian determines a {\em period mapping} $\IJac: {\mathcal M}_{4,3}\to {\mathcal A}_{30}$ from the moduli space of smooth quartic $3$-folds to the moduli space of $30$-dimensional principally polarized abelian varieties.  $\IJac$ is a holomorphic map between quasiprojective varieties.  Since the target ${\mathcal A}_{30}$ is the quotient of a bounded symmetric domain by an arithmetic lattice, 
Theorem 3.10 of Borel \cite{Bo} gives that $\IJac$ is in fact a morphism.  Let ${\mathcal P}\subset{\mathcal A}_{30}$ denote the subvariety consisting of products of Jacobians of smooth curves.   Then 
$\IJac^{-1}({\mathcal P})$ is a subvariety of ${\mathcal M}_{4,3}$.  Theorem \ref{Theorem:IJ1} implies that the inclusion 
$\IJac^{-1}({\mathcal P})\subset {\mathcal M}_{4,3}$ is strict.  The irreducibility of ${\mathcal M}_{4,3}$ then gives: 

\begin{corollary}[{\bf Irrationality is general}]
\label{cor:general}
The general smooth quartic $3$-fold is irrational.\footnote{In other words, there is a subvariety $V\subsetneq {\mathcal M}_{4,3}$ such that each 
$X\in  {\mathcal M}_{4,3}\setminus V$ is irrational.}
\end{corollary}

\bigskip
\noindent
{\bf Context. }Corollaries \ref{cor:Klein1} and \ref{cor:general} are not new.  Iskovskih-Manin \cite{IM} proved in 1971 that {\em any} smooth quartic $3$-fold $X$ is irrational.  In contrast, Segre had constructed in \cite{Se} (see also \S 9 of \cite{IM}) examples 
of such $X$ that are {\em unirational}: there is a dominant rational map $\Pb^3\dashrightarrow X$.    Iskovskih-Manin prove their theorem by 
developing the ``method of maximal singularities'' to prove that any birational map $X\dashrightarrow X$ has finite order, and noting that this is of course not true for $\Pb^3$.  This initiated the modern theory of 
birational superrigidity; see, e.g. Cheltsov \cite{Ch} for a survey and details.  More recently, 
Colliot-Th\'{e}l\`{e}ne-Pirutka \cite{CP}, building on a method of Voisin 
using the Chow group of $0$-cycles, proved that the very general smooth quartic $3$-fold is not stably rational.

Around the same time as Iskovskih-Manin, Clemens-Griffiths \cite{CG} used their criterion mentioned above to prove that any smooth {\em cubic} $3$-fold $Y$ is irrational, even though any such $Y$ is unirational.  The bulk of their 
proof is showing that $\IJac(Y)$ is not a product of Jacobians of curves.

Intermediate Jacobians have been used (via the 
Clemens-Griffiths criterion) to prove irrationality for  many $3$-folds, but not (as far as we can tell) for smooth quartic $3$-folds; 
see Beauville's survey \cite{B1}, in particular the table on page $6$.    The proof of Theorem \ref{Theorem:IJ1} uses the symmetry of $X$ in a crucial way, and 
follows an idea of Beauville (see \cite{B1} ( \S 3.3), \cite{B2,B3}, and also Zarhin \cite{Z}) to whom we owe an intellectual debt.  Note that \cite{B3} is about quartic $3$-folds, but they are singular; this is, however, enough to deduce that the intermediate Jacobian of the general quartic $3$-fold is 
not a product of Jacobians (cf.\ Lemma 5.6.1 of \cite{B4}).  The point of the present paper is that we find an explicit such example.  The hardest 
part was actually finding the example.  It may also be worth noting that the proofs of all of the results in this paper use technology available already in 1972.

\bigskip
\noindent
{\bf Acknowledgements. }I thank Nick Addington and Jeff Achter for useful discussions; and Ronno Das, J\'{a}nos Koll\'{a}r and an anonymous referee for corrections on an earlier version of this paper. I am also extremely grateful to Curt McMullen, whose many useful comments on an earlier version of this paper greatly improved its exposition.

\section{Proof of Theorem \ref{Theorem:IJ1}}
\label{section:proof}

In this note we always work in the category of principally polarized abelian varieties.  The polarization is crucial for the proofs that follow.  For any p.p.a.v $A$, denote by $\Aut(A)$ the group of automorphisms of $A$ respecting the polarization; in particular $\Aut(A)$ is finite (see, e.g.\ \cite{BL}, Corollary 5.1.9).  Without the polarization this is no longer true: consider the automorphism of 
$A:=\Cb^2/\Zb[i]^2$ induced by $(z,w)\mapsto (2z+w,z+w)$, which is an infinite order algebraic automorphism of $A$.

Recall that the Jacobian ${\rm Jac}(C)$ of a smooth, projective curve $C$ is a p.p.a.v., with polarization induced by the intersection pairing on $H_1(C;\Zb)$. We will need the following.

\begin{lemma}
\label{lemma:Riemann}
Let $C$ be any smooth, projective curve of genus $g\geq 2$ and let 
${\rm Jac}(C)$ denote its Jacobian.  Then for any odd order subgroup $G\subset \Aut({\rm Jac}(C))$ the following hold.
\begin{enumerate}
\item Any cyclic subgroup of $G$ has order at most $4g+2$.
\item If $g\geq 4$ and if $G$ is {\em metacyclic} (meaning that $G$ has a cyclic normal subgroup $N\lhd G$ such that $G/N$ is cyclic) 
then $|G|\leq 9(g-1)$.
\end{enumerate}
\end{lemma}

\begin{proof}
For any smooth projective curve $C$ of genus $g\geq 2$ the natural map $\rho: \Aut(C)\to\Aut({\rm Jac}(C))$ is injective; see, e.g. \cite{FM}, Theorem 6.8. The classical Torelli theorem gives that $\rho$ is surjective if 
$C$ is hyperelliptic, and otherwise 
$\Aut({\rm Jac}(C))=\rho(\Aut(C))\times \{\pm {\rm I}\}$.  
Since $|G|$ is assumed to be 
odd, there is a subgroup $\tilde{G}\subset\Aut(C)$ such that $\rho:\tilde{G}\to G$ is an isomorphism.     Both parts of the lemma now follow from the corresponding statements for subgroups of $\Aut(C)$; see e.g. Theorem 7.5 of \cite{FM} (which is classical) and Proposition 4.2 of \cite{Sch}, a result of Schweizer.   
\end{proof}

\begin{proof}[Proof of Theorem \ref{Theorem:IJ1}]

Let $X$ be the Klein quartic $3$-fold, and let $\zeta:=e^{2\pi i/61}$ 
The group $G:=\Zb/61\Zb\rtimes \Zb/5\Zb$ acts on $X$ by automorphisms via the maps
\[
\begin{array}{l}
\phi([x_0:x_1:x_2:x_3:x_4]):=[x_0:\zeta^{3}x_1:\zeta^{-6}x_2:\zeta^{21}x_3:\zeta x_4]\\
\\
\psi([x_0:x_1:x_2:x_3:x_4]):=[x_1:x_2:x_3:x_4:x_0]
\end{array}
\]
of order $61$ and $5$, respectively \footnote{The somewhat surprisingly large order automorphism $\phi$ is based on a similar automorphism on the Klein quartic; see Theorem 3.7 of \cite{GL} for an even more general discussion.}; in fact $G\cong \Aut(X)$ (see \cite{GLMV}, Theorem B), but we will not need this.  For any smooth, degree $d\geq 3$ hypersurface in $\Pb^n, n>1$, the action of $\Aut(X)$ 
on $H^n(X;\Zb)$ is faithful (see, e.g.,\! Chap.1, Cor.\!\! 3.18 of \cite{H}).  Since in addition $\Aut(X)$ preserves the Hodge decomposition of $H^3(X;\Cb)$, it follows that $\Aut(X)$, hence $G$, acts faithfully on $\IJac(X)$ by 
p.p.a.v automorphisms.

Suppose that $X$ is rational.  The Clemens-Griffiths criterion 
gives an isomorphism of p.p.a.v.:
\begin{equation}
\label{eq:CG}
A:=\IJac(X)\cong A_1^{n_1}\times\cdots\times A_r^{n_r}
\end{equation}
where each $A_i:={\rm Jac}(C_i)$ is the Jacobian of a smooth, projective curve $C_i$ and where $A_i\not\cong A_j$ if $i\neq j$.  Now, say that a  p.p.a.v $A$ is {\em irreducible} if any morphism $A'\to A$ of p.p.a.v is $0$ or an isomorphism.  Corollary 3.23 of \cite{CG} states that a p.p.a.v 
A is irreducible if and only if its theta divisor is; it follows that each $A_i$ is irreducible since each $C_i$ is smooth.  Corollary 3.23 of \cite{CG} also states that the decomposition of a p.p.a.v into a direct sum of irreducible p.p.a.v is unique.  It follows that for each $i$: 
\[\Aut(A_i^{n_i})\cong\Aut(A_i)^{n_i}\rtimes S_{n_i}\]
Note that this kind of statement is no longer true if we consider abelian varieties without polarizations.  

Back to the case $A=\IJac(X)$.  The group $G$ acts on $A$ as p.p.a.v. automorphisms.   The uniqueness of the decomposition \eqref{eq:CG} implies that each $A_i^{n_i}$ is $G$-invariant, and that this $G$-action is given by a representation 
\begin{equation}
\label{eq:psi1}
\psi_i:G\to \Aut(A_i^{n_i})\cong\Aut(A_i)^{n_i}\rtimes S_{n_i}.
\end{equation}

Now 
\begin{equation}
\label{eq:dim}
30=\dim(A)=\sum_{i=1}^r n_i\dim(A_i).
\end{equation}

Since the $G$-action on $A$ is faithful and $\Zb/61\Zb$ has prime order, 
there exists some $i$ (after re-labeling assume $i=1$) so that 
$\Zb/61\Zb$ acts faithfully on $A_1^{n_1}$; that is, $\psi_1$ in \eqref{eq:psi1} is injective.  Since $n_1< \dim(A)<61$ and so $S_{n_1}$ has no element of order $61$, the composition of $\psi_1$ with the projection to $S_{n_1}$ is trivial, so that $\Zb/61\Zb$ acts 
on each 
direct factor $A_1$ of $A_1^{n_1}$.

Fix such a direct factor $B\cong A_1$ on which $\Zb/61\Zb$ acts faithfully (such a factor must exist since $\Zb/61\Zb$ acts faithfully on $A_1^{n_1}$, as noted above).  Recall that $B\cong A_1\cong {\rm Jac}(C_1)$ for some smooth projective curve $C_1$ of genus $g\geq 1$.  Note that in fact $g\geq 2$ since otherwise $\dim(B)=1$ and so $A_1$ does not admit a p.p.a.v.  automorphism of order $>6$.   Thus Lemma~\ref{lemma:Riemann}(1) applies, giving 
\[61\leq 4\cdot{\rm genus}(C_1)+2=4\dim(B)+2\]
and so $\dim(A_1)=\dim(B)={\rm genus}(C_1)\geq 15$.  Again by the orbit-stabilizer theorem, the orbit of $B$ in the set of direct factors of $A_1^{n_1}$ under the prime order subgroup $\Zb/5\Zb\subset G$ has $1$ or $5$ elements.  Since $\dim(B)={\rm genus}(C_1)\geq 15$ and 
$n_1\cdot {\rm genus}(C_1)\leq 30$, the latter is not possible; that is, $B$ is $\Zb/5\Zb$-invariant, and so $G$-invariant.

Now, the definition of $\phi$ and $\psi$ above give that $G\cong \Zb/61\Zb\rtimes \Zb/5\Zb$ is a nontrivial semidirect product; that is, $G$ is not a direct product. For any homomorphism $\mu:C\rtimes D\to E$ of a {\em nontrivial} semidirect product of finite groups to any group, if $\mu$ is 
trivial on $D$ then it is not injective on $C$, hence it is trivial if $C$ is simple.  Since the $\Zb/61\Zb$-action on $B$ is faithful, it follows that the $\Zb/5\Zb$ action on $B$ is faithful. From this it follows that the $G$-action on $B$ is faithful (consider the kernel $K$ of the $G$-action, and note that $K\cap \Zb/61\Zb=0$ and so $K<\Zb/5\Zb$, so that $K$ is trivial).

Note that 
\begin{equation}
\label{eq:last}
|G|=61\cdot 5=305 > 261=9\cdot (30-1) > 9({\rm genus}(C_1)-1).
\end{equation}

Since ${\rm genus}(C_1)\geq 15\geq 4$ and since $G$ is metacyclic,  Lemma~\ref{lemma:Riemann}(2) applies. Its conclusion contradicts \eqref{eq:last}. Thus $X$ is not rational.

\end{proof}

\begin{remark}
One might hope to replace the use of Lemma \ref{lemma:Riemann}(2) by something simpler, such as the Hurwitz bound $|\Aut(C)|\leq 84(g-1)$.  However, a quick check of the numerology shows that this is not enough to obtain a contradiction. \end{remark}

\bigskip{\noindent
Dept. of Mathematics\\
University of Chicago\\
E-mail: bensonfarb@gmail.com

\end{document}